\documentclass[a4paper, 10pt, twoside]{article}

\usepackage{amsmath, amscd, amsfonts, amssymb, amsthm, latexsym, color, framed}
\usepackage[hyphens]{url}

\usepackage{graphicx}
\usepackage[left=1in, right=1in, top=1.5in, bottom=1.5in, includefoot, headheight=13.6pt]{geometry}

\input{xy}
\xyoption{all}

\usepackage[english, francais]{babel}
\usepackage[T1]{fontenc}

\newtheorem{theorem}{Theorem}[section]
\newtheorem{lemma}[theorem]{Lemma}
\newtheorem{e-proposition}[theorem]{Proposition}
\newtheorem{corollary}[theorem]{Corollary}
\newtheorem{e-definition}[theorem]{Definition\rm}
\newtheorem{remark}{\it Remark\/}

\newtheorem{theoreme}{Th\'eor\`eme}[section]

\newtheorem{corollaire}[theoreme]{Corollaire}

\usepackage{fancyhdr}

\newcommand{\bb}{\mathbb}

\newcommand{\blob}{\bullet}

\newcommand{\comment}[1]{}

\newcommand{\into}{\hookrightarrow}
\newcommand{\isoto}{\stackrel{\simeq}{\to}}

\newcommand{\op}{\operatorname}

\newcommand{\sub}[1]{{\mbox{\scriptsize #1}}}

\newcommand{\To}{\longrightarrow}

\newcommand{\xto}{\xrightarrow}

\renewcommand{\hat}{\widehat}

\renewcommand{\ker}{\operatorname{Ker}}
\renewcommand{\projlim}{\varprojlim}

\usepackage{sectsty}
\sectionfont{\bf\normalsize}
\subsubsectionfont{\it\normalsize}
\chapterfont{\large\sc\centering}
\chaptertitlefont{\LARGE\centering}
\partfont{\centering}

\begin{document}

\selectlanguage{english}
\title{A case of the deformational Hodge conjecture via a pro Hochschild--Kostant--Rosenberg theorem}

\selectlanguage{english}
\author{Matthew Morrow}
\date{\normalsize Appears in {\em Comptes Rendus Mathématique}, vol.~352, issue 03, 173--177, 2014.}


\maketitle
\noindent\hrulefill
\begin{abstract}
\selectlanguage{english}
Following ideas of Bloch, Esnault, and Kerz, we establish the deformational part of Grothendieck's variational Hodge conjecture for proper, smooth schemes over $K[[t]]$, where $K$ is an algebraic extension of $\bb Q$. The main tool is a pro Hochschild--Kostant--Rosenberg theorem for Hochschild homology.

\vskip 0.5\baselineskip

\selectlanguage{francais}
\begin{center}{\bf R\'esum\'e}\end{center}
{\bf Un cas de la conjecture de Hodge déformationnelle via un théorème de Hochschild--Kostant--Rosenberg pro.}
En suivant des idées de Bloch, Esnault et Kerz, nous établissons la partie déformationnelle de la conjecture de Hodge variationnelle pour les schémas propres et lisses sur $K[[t]]$, où $K$ est une extension algébrique de $\bb Q$. L'outil principal est un théorème de Hochschild--Kostant--Rosenberg pro pour l'homologie de Hochschild.
\end{abstract}
\noindent\hrulefill

\selectlanguage{francais}
\section*{Version fran\c{c}aise abr\'eg\'ee}
\subsubsection*{Introduction et présentation du résultat principal}
S.~Bloch, H.~Esnault et M.~Kerz dans \cite{BlochEsnaultKerz2012} ont proposé de décomposer la conjecture de Hodge variationnelle $p$-adique de Fontaine--Messing en deux parties : premièrement une partie formelle déformationnelle et secondement une partie algébrisationnelle. Sous une hypothèse faible que la caractéristique est grande par rapport à la dimension, ils ont établi la partie déformationnelle. A.~Beilinson \cite{Beilinson2013} a donné récemment une nouvelle preuve. Le but de cette note est de démontrer une partie déformationnelle analogue pour la conjecture de Hodge variationnelle de Grothendieck pour les schémas propres et lisses sur $ K[[t]]$, où $K$ est une extension algébrique de $\bb Q $. De manière remarquable, nous la déduisons facilement d'un nouveau théorème de Hochschild--Kostant--Rosenberg (HKR) pro en homologies de Hochschild et cyclique.

Pour présenter le résultat principal, considérons $A=K[[t]]$, où $K$ est une extension algébrique de $\bb Q$, et soit $X$ un schéma propre et lisse sur $A$ ; notons $Y$ la fibre spéciale et $Y_r=Y\otimes_AA/t^rA$ son $r$-ième épaississement infinitésimal. En utilisant le théorème de HKR pro que nous allons décrire ci-dessous, nous construisons le diagramme commutatif (\ref{main_commutative_diagram}). Dans ce diagramme, $F^p_\sub{flat}H^{2p}_\sub{dR}(Y/K)$ désigne les classes de $H^{2p}_\sub{dR}(Y/K)$ dont les relèvements plats à $H^{2p}_\sub{dR}(X/A)$ appartiennent à $F^pH^{2p}_\sub{dR}(X/A)$, et $\Omega_{(Y_r,Y)/K}^\blob:=\ker(\Omega_{Y_r/K}^\blob\to \Omega_{Y/K}^\blob)$. En particulier, le carré gauche du diagramme (\ref{main_commutative_diagram}) est cocartesien, ce qui implique immédiatement le résultat suivant:

\begin{theoreme}\label{main_theorem_french}
Dans la situation ci-dessus,\comment{les conditions suivantes sont équivalentes pour tout $z\in K_0(Y)$:
\begin{enumerate}
\item on peut relever $z$ à $\projlim_rK_0(Y_r)$ ;
\item la classe de de Rham $ch(z)\in F^pH_\sub{dR}^{2p}(Y/K)$ appartient à $F^p_\sub{flat}H_\sub{dR}^{2p}(Y/K)$ pour tout entiers $p=0,\dots,\dim Y$.
\end{enumerate}
}
soit $z\in K_0(Y)$. Alors on peut relever $z$ à $\projlim_rK_0(Y_r)$ si et seulement si la classe de de Rham $ch(z)\in F^pH_\sub{dR}^{2p}(Y/K)$ appartient à $F^p_\sub{flat}H_\sub{dR}^{2p}(Y/K)$ pour tout entiers $p=0,\dots,\dim Y$.
\end{theoreme}

Le théorème est analogue au résultat $p$-adique de Bloch, Esnault et Kerz mentionné ci-dessus, et on peut l'appeler la partie déformationnelle de la conjecture de Hodge variationnelle pour $X$.

\subsubsection*{Le théorème de Hochschild--Kostant--Rosenberg pro et la démonstration du résultat principal}\label{section_HKR_french}
Expliquons maintenant  le théorème de HKR pro que nous utilisons pour démontrer le théorème principal. Pour simplifier, considérons un corps $k$ de caractéristique nulle, même si les résultats sont vrais en plus grande généralité. Pour un $k$-schéma $X$, pas nécessairement de type fini, on note $H_\sub{dR}^n(X/k)$ l'hypercohomologie du complexe de de Rham $\Omega_{X/k}^\bullet$ (pas la cohomologie de de Rham algébrique de Hartshorne).

D'abord, rappelons le théorème de HKR classique de 1962 \cite[Thm.~3.4.4]{Loday1992}. Il dit que si $R$ est une $k$-algèbre lisse alors l'homologie de Hochschild de $R$ sur $k$ dégénère naturellement aux différentielles de K\"ahler via l'application antisymétrisation : $\Omega_{R/k}^n\isoto HH_n^k(R)$ pour tout $n\ge0$. Des conséquences formelles de ce résultat sont des descriptions similaires des homologies cyclique, cyclique négative et cyclique périodique en termes de complexes de de Rham (tronqués) : voir (\ref{consequences_of_HKR}). Grâce à la désingularisation de Neron--Popescu\comment{\cite{Popescu1985,Popescu1986}}, on a les isomorphismes (\ref{consequences_of_HKR}) même si $R$ est une $k$-algèbre régulière. Plus généralement, si $X$ est un schéma de dimension de Krull finie qui admet un revêtement fini par les spectres des anneaux réguliers, alors les techniques courantes de descente \cite{Weibel1997} donnent les isomorphismes~(\ref{global_consequences_of_HKR}).

Tournons-nous maintenant vers le théorème de HKR pro. Plusieurs versions différentes de ce résultat sont apparues récemment, par exemple \cite[Thm.~3.2]{Cortinas2009} ou \cite{Krishna2010}, mais pour cette note la version forte suivante, qui est une conséquence de notre travail \cite{Morrow_pro_H_unitality}, est nécessaire :

\begin{theoreme}[{Théorème de HKR Pro \cite[Thm.~0.8]{Morrow_pro_H_unitality}}]\label{Pro_HKR_french}
Si $R$ est une $k$-algèbre réguleière et $I$ un idéal de $R$, alors pour tout $n\ge 0$ le morphisme canonique $\{\Omega_{(R/I^r)/k}^n\}_r\To\{HH_n^k(R/I^r)\}_r$ de groupes abéliens pro est un isomorphisme.
\end{theoreme}

Comme le théorème de HKR classique, le théorème de HKR pro a des conséquences pour les homologies cyclique, cyclique négative et cyclique périodique, ainsi que pour les schémas ; en particulier, on obtient le corollaire suivant :

\begin{corollaire}\label{corollary_pro_HKR_french}
Soit $X$ un schéma de dimension de Krull finie qui admet un revêtement fini par les spectres des anneaux réguliers, et soit $Y\into X$ une immersion fermée. Alors pour tout $n\in \bb Z$ il existe un isomorphisme naturel $\{HN_n^k(Y_r)\}_r\cong\bigoplus_{p\in\bb Z}\{\bb H^{2p-n}(X,\Omega_{Y_r/k}^{\ge p})\}_r$ de groupes abéliens pro, où $Y_r$ désigne le $r$-ième épaississement infinitésimal de $Y$ dans $X$.
\end{corollaire}

Finalement, expliquons brièvement la construction du diagramme commutatif (\ref{main_commutative_diagram}), dont le théorème principal décole. Soient $K, A, X, Y, Y_r$ comme dans l'introduction. Grâce à la théorie du caractère de Chern de la $K$-théorie à l'homologie cyclique négatif, on a un diagramme commutatif (\ref{chern_character}) pour tout $n\ge0$, et le théorème de HKR classique et le corollaire \ref{corollary_pro_HKR_french} (avec $k=\bb Q$) impliquent que le ligne du bas de ce diagramme est égale à
\[\textstyle\bigoplus_p\{\bb H^{2p-n}(Y_r,\Omega_{Y_r/K}^{\ge p})\}_r\stackrel{i}{\To}\bigoplus_pF^pH^{2p-n}_\sub{dR}(Y/K)\To\bigoplus_p\{\bb H^{2p-(n-1)}(Y_r,\Omega_{(Y_r,Y)/K}^{\ge p})\}_r\]
Or, en utilisant la théorie de la connexion de Gauss--Manin (voir (\ref{GM_commutative_diagram}) et le lemme \ref{lemma_flat_filtration}), on voit que l'image de $i$ est $\bigoplus_pF^p_\sub{flat}H_\sub{dR}^{2p-n}(Y/K)$. Le cas $n=0$ prouve l'existence du diagramme (\ref{main_commutative_diagram}).

\selectlanguage{english}
\section{Introduction and statement of the main result}
In \cite{BlochEsnaultKerz2012}, S.~Bloch, H.~Esnault, and M.~Kerz proposed decomposing Fontaine--Messing's $p$-adic variational Hodge conjecture into two parts: firstly a formal deformational part and secondly an algebrizational part. Under a mild assumption that the characteristic is large compared to the dimension, they proved the deformational part. A new proof has very recently been given by A.~Beilinson \cite{Beilinson2013}. The aim of this note is to prove an analogous deformational part of Grothendieck's variational Hodge conjecture for proper, smooth schemes over $K[[t]]$, where $K$ is an algebraic extension of $\bb Q$. Remarkably, this will be quickly deduced from a recent pro Hochschild--Kostant--Rosenberg (HKR) theorem in Hochschild and cyclic homology.

To precisely state the main result, let $A=K[[t]]$, where $K$ is an algebraic extension of $\bb Q$, and let $X$ be a proper, smooth scheme over $A$; let $Y$ denote the special fibre, and write $Y_r=Y\otimes_AA/t^rA$ for its $r^\sub{th}$ infinitesimal thickening. Using the pro HKR theorem which will be described in section \ref{section_HKR}, we will construct a natural commutative diagram with exact rows:
\begin{equation}
\xymatrix@C=1cm@R=7mm{
\projlim_rK_0(Y_r)\ar[r]\ar[d]& K_0(Y)\ar[d]^{ch}\ar[r]& \projlim_rK_{-1}(Y_r,Y)\ar[d]^\cong\\
\bigoplus_pF^p_\sub{flat}H^{2p}_\sub{dR}(Y/K)\ar@{^(->}[r]& \bigoplus_pF^pH^{2p}_\sub{dR}(Y/K)\ar[r]& \bigoplus_p\projlim_r\bb H^{2p+1}(Y_r,\Omega_{(Y_r,Y)/K}^{\ge p})
}\label{main_commutative_diagram}
\end{equation}
Here $F^p_\sub{flat}H^{2p}_\sub{dR}(Y/K)$ denotes the classes of $H^{2p}_\sub{dR}(Y/K)$ whose flat lift to $H^{2p}_\sub{dR}(X/A)$ belongs to $F^pH^{2p}_\sub{dR}(X/A)$ (see section \ref{section_main} for details), and $\Omega_{(Y_r,Y)/K}^\blob:=\ker(\Omega_{Y_r/K}^\blob\to \Omega_{Y/K}^\blob)$. In particular the left square is cocartesian, which immediately yields the following:

\begin{theorem}\label{main_theorem}
In the situation above, the following are equivalent for any $z\in K_0(Y)$:
\begin{enumerate}
\item $z$ lifts to $\projlim_rK_0(Y_r)$.
\item The de Rham class $ch(z)\in F^pH_\sub{dR}^{2p}(Y/K)$ belongs to $F^p_\sub{flat}H_\sub{dR}^{2p}(Y/K)$ for $p=0,\dots,\dim Y$.
\end{enumerate}
\end{theorem}

The theorem may be called the deformation part of Grothendieck's variational Hodge conjecture for $X$; it is precisely analogous to the aforementioned $p$-adic result of Bloch, Esnault, and Kerz.

\section{The pro Hochschild--Kostant--Rosenberg theorem}\label{section_HKR}
Here we explain the pro HKR theorem which will be used to deduce the main theorem. For simplicity we fix a characteristic zero field $k$, although the results hold more generally; to prove the main theorem we will take $k=\bb Q$. Given a $k$-scheme $X$, not necessarily of finite type, we denote by $H_\sub{dR}^n(X/k)$ the hypercohomology of the de Rham complex $\Omega_{X/k}^\bullet$ (not Hartshorne's algebraic de Rham cohomology).

We first recall the classical 1962 HKR theorem \cite[Thm.~3.4.4]{Loday1992}. It states that if $R$ is a smooth $k$-algebra, then the Hochschild homology of $R$ over $k$ naturally degenerates to K\"ahler differentials via the antisymmetrization map: $\Omega_{R/k}^n\isoto HH_n^k(R)$ for all $n\ge0$. Formal consequences of this are similar descriptions of the cyclic, negative cyclic, and periodic cyclic homologies in terms of (truncated) de Rham complexes:
\begin{equation}\textstyle
HC_n^k(R)\cong\bigoplus_{p=0}^nH^{2p-n}(\Omega_{R/k}^{\le p})\comment{=\Omega_{R/k}^n/d\Omega_{R/k}^{n-1}\oplus\bigoplus_{p=0}^{n-1}H_\sub{dR}^{2p-n}(R/k)\\}\quad
HN_n^k(R)\cong\bigoplus_{p\in\bb Z}H^{2p-n}(\Omega_{R/k}^{\ge p})\quad
HP_n^k(R)\cong\bigoplus_{p\in\bb Z}H^{2p-n}_\sub{dR}(R/k).
\label{consequences_of_HKR}
\end{equation}
Using Neron--Popescu desingularisation\comment{\cite{Popescu1985,Popescu1986}}, these isomorphisms remain valid whenever $R$ is a regular $k$-algebra. More generally, if $X$ is a finite Krull dimensional $k$-scheme which has a finite cover by the spectra of regular rings, then these results globalise by the usual descent methods \cite{Weibel1997} to give \begin{equation}\textstyle HC_n^k(X)\cong\bigoplus_{p\in\bb Z}\bb H^{2p-n}(X,\Omega_{X/k}^{\le p})\quad HN_n^k(X)\cong\bigoplus_{p\in\bb Z}\bb H^{2p-n}(X,\Omega_{X/k}^{\ge p})\quad
HP_n^k(X)\cong\bigoplus_{p\in\bb Z}H^{2p-n}_\sub{dR}(X/k).\label{global_consequences_of_HKR}
\end{equation}

Next we turn to the pro HKR theorem. Various versions of this result have recently appeared, e.g., \cite[Thm.~3.2]{Cortinas2009} \cite{Krishna2010}, but for this note the following strong version, which follows from the author's work \cite{Morrow_pro_H_unitality}, is required:

\begin{theorem}[{Pro HKR Theorem \cite[Thm.~0.8]{Morrow_pro_H_unitality}}]\label{Pro_HKR}
If $R$ is a regular $k$-algebra and $I$ is any ideal of $R$, then for all $n\ge0$ the canonical map of pro abelian groups $\{\Omega_{(R/I^r)/k}^n\}_r\To\{HH_n^k(R/I^r)\}_r$ is an isomorphism.
\end{theorem}
\begin{proof}
For each $r\ge1$, there is a natural map $HH_n^k(R,R/I^r)\to HH_n^k(R/I^r)$ from the Hoschschild homology of $R$ with coefficients in $R/I^r$ to the Hochschild homology of $R/I^r$. Taking the limit over $r$, these assemble into a map of pro $R$-modules $\{HH_n^k(R,R/I^r)\}_r\to \{HH_n^k(R/I^r)\}_r$, which can be shown to be an isomorphism \cite[Lem.~3.7]{Morrow_pro_H_unitality}. Moreover, since $R$ is regular, the classical HKR theorem above implies that $HH_n^k(R,R/I^r)\cong \Omega_{R/k}^n\otimes_RR/I^r$. Finally, the isomorphism $\{\Omega_{R/k}^n\otimes_RR/I^r\}_r\cong\{\Omega_{(R/I^r)/k}^n\}$ is an easy consequence of the inclusion $d(I^{2r})\subseteq I^r\Omega_{R/k}^1$.
\end{proof}

In other words, even though the classical HKR theorem does not apply to the non-regular rings $R/I^r$, it applies in the limit over powers of $I$. This has exactly the same formal consequences for $HC$, $HN$, $HP$, and for the global setting as the classical HKR theorem. In particular, one obtains the following:

\begin{corollary}\label{corollary_pro_HKR}
Let $X$ be a finite Krull dimensional $k$-scheme which has a finite cover by the spectra of regular rings, and let $Y\into X$ be a closed subscheme. Then for each $n\in\bb Z$ there is a natural isomorphism of pro abelian groups, $\{HN_n^k(Y_r)\}_r\cong\bigoplus_{p\in\bb Z}\{\bb H^{2p-n}(X,\Omega_{Y_r/k}^{\ge p})\}_r$, where $Y_r$ denotes the $r^\sub{th}$ infinitesimal thickening of $Y$ inside $X$.
\end{corollary}

\section{Deformational part of the variational Hodge conjecture: proof of Theorem \ref{main_theorem}}\label{section_main}
As in the introduction, let $K$ be a characteristic zero field, let $A=K[[t]]$, let $X$ be a proper, smooth scheme over $A$, let $Y$ denote the special fibre, and write $Y_r=Y\otimes_AA/t^rA$ for its $r^\sub{th}$ infinitesimal thickening. This notation is fixed for the remainder of the note. The following discussion and lemma concerning the Gauss--Manin connection and flat filtration are presumably well-known to experts and probably even essentially contained in Bloch's seminal paper on the subject \cite{Bloch1972}.

The short exact sequences of complexes of coherent sheaves \begin{equation}0\To \Omega^1_{A_r/K}\otimes_A\Omega^{\bullet-1}_{Y_r/A_r}\To\Omega^\bullet_{Y_r/K}\To\Omega^\bullet_{Y_r/A_r}\To 0\label{GM_complexes}\end{equation} give rise to long exact sequences of finite dimensional hypercohomology groups, so we may take $\projlim_r$ to obtain the formal Gauss--Manin exact sequence
\begin{equation}\cdots\To\projlim_rH_\sub{dR}^n(Y_r/K)\To H_\sub{dR}^n(X/A)\xto{\nabla}\hat\Omega^1_{A/K}\otimes_A H_\sub{dR}^n(X/A)\To\cdots.\label{GM}\end{equation} Here we have identified $H_\sub{dR}^n(X/A)$ with $\projlim_rH_\sub{dR}^n(Y_r/A_r)$ using Grothendieck's formal functions theorem \cite[Cor.~4.1.7]{EGA_III_I}. Moreover, $H_\sub{dR}^n(Y_r/K)=H_\sub{dR}^n(Y/K)$ for all $r$ by the Poincar\'e lemma \cite[Corol.~9.9.3]{Weibel1994}, and so (\ref{GM}) breaks into short exact sequences, identifying $H_\sub{dR}^n(Y/K)$ with the so-called {\em flat/horizontal} classes $H_\sub{dR}^n(X/A)^\nabla:=\ker\nabla$. Diagrammatically,
\begin{equation}
\xymatrix@C=1cm@R=0.7cm{
0\ar[r]&\projlim_rH_\sub{dR}^n(Y_r/K)\ar[r]\ar[dr]^\cong_\Phi& H_\sub{dR}^n(X/A)\ar[d]\ar[r]^{\nabla\qquad}&\hat\Omega^1_{A/K}\otimes_A H_\sub{dR}^n(X/A)\ar[r]&0\\
&&H_\sub{dR}^n(Y/K)&&
}\label{GM_commutative_diagram}
\end{equation}
Set $F^p_\sub{flat}H^n_\sub{dR}(Y/K):=\{x\in H_\sub{dR}^n(Y/K):\Phi^{-1}(x)\in F^pH^n_\sub{dR}(X/A)\}\subseteq F^pH^n_\sub{dR}(Y/K)$. We will need the following alternative description of this flat filtration:

\begin{lemma}\label{lemma_flat_filtration}
$\displaystyle F^p_\sub{flat}H^n_\sub{dR}(Y/K)=\op{Im}\Big(\projlim_r\bb H^n(Y_r,\Omega_{Y_r/K}^{\ge p})\To\bb H^n(Y,\Omega_{Y/K}^{\ge p})=F^pH^n_\sub{dR}(Y/K)\Big).$
\end{lemma}
\begin{proof}
Naively truncating (\ref{GM_complexes}) in degrees $\ge p$, and using degeneration of the Hodge-to-de-Rham spectral sequence for $X\to\operatorname{Spec} A$ to identify $\bb H^n(X,\Omega_{X/A}^{\ge p})$ with $F^pH^n_\sub{dR}(X/A)$, we see that we may add
\[\xymatrix@C=1cm@R=0.5cm{
\cdots\ar[r]& \projlim_r\bb H^n(Y_r,\Omega_{Y_r/K}^{\ge p})\ar[r]\ar[d]& F^pH_\sub{dR}^n(X/A)\ar@{^(->}[d]\ar[r]^{\nabla\qquad}& \hat\Omega_{A/K}^1\otimes_A F^pH_\sub{dR}^n(X/A)\ar[r]\ar@{^(->}[d]&\cdots\\
&&&
}\]
to the top of the commutative diagram (\ref{GM_commutative_diagram}). A quick diagram chase now shows that if $x$ is in $H_\sub{dR}^n(Y/K)$, then $\Phi^{-1}(x)$ is in $F^pH_\sub{dR}^n(X/A)$ if and only if $x$ is in the image of $\projlim_r\bb H^n(Y_r,\Omega_{Y_r/K}^{\ge p})$, as required.
\end{proof}
\comment{
\begin{remark}
The flat filtration $F^*_\sub{flat}H^n_\sub{dR}(Y/k)$ may also be described as follows: \[F^p_\sub{flat}H^n_\sub{dR}(Y/k)=\op{Im}\big(\projlim_r\bb H^n(Y_r,\Omega_{Y_r}^{\ge p})\To\bb H^n(Y,\Omega_Y^{\ge p})=F^pH^n_\sub{dR}(Y)\big)\] This is obtained as followed: stupidly truncate the complexes occurring in (\dag) and compare the resulting long exact sequence with the previous diagram:
\[\xymatrix@C=1.2cm@R=5mm{
\cdots\ar[r]& \projlim_r\bb H^n(Y_r,\Omega_{Y_r/k}^{\ge p})\ar[r]\ar[d]& F^pH_\sub{dR}^n(X/A)\ar@{^(->}[d]\ar[r]^{\nabla\qquad}& \hat\Omega_{A/k}^1\otimes_A F^pH_\sub{dR}^n(X/A)\ar[r]\ar@{^(->}[d]&\cdots\\
0\ar[r]&\projlim_rH_\sub{dR}^n(Y_r/k)\ar[r]\ar[dr]^\cong_\Phi& H_\sub{dR}^n(X/A)\ar[d]\ar[r]^{\nabla\qquad}&\hat\Omega^1_{A/k}\otimes_A H_\sub{dR}^n(X/A)\ar[r]& 0\\
&&H_\sub{dR}^n(Y/k)&&
}\]
(Degeneration of the Hodge-to-de-Rham spectral sequence for $X\to\Spec A$ has been used to identify $\bb H^n(X,\Omega_{X/A}^{\ge p})$ with $F^pH^n_\sub{dR}(X/A)$.) A quick diagram chase shows that if $x\in H_\sub{dR}^n(Y/k)$, then $\Phi^{-1}(x)$ is in $F^pH_\sub{dR}^n(X/A)$ if and only if $x$ is in the image of $\projlim_r\bb H^n(Y_r,\Omega_{Y_r/k}^{\ge p})$, as required.
\end{remark}
}

We may now prove the existence of the commutative diagram (\ref{main_commutative_diagram}) promised in the introduction; in fact we work with an arbitrary $K$-group $K_n$, rather than only $K_0$, and we will comment on the case $n>0$ after the proof:

\begin{theorem}\label{proof_of_main_diagram}
Assume that $K$ is algebraic over $\bb Q$. Then, for all $n\ge0$, there is a natural commutative diagram with exact rows
\[\xymatrix@C=8mm{
\projlim_rK_n(Y_r)\ar[r]\ar[d]& K_n(Y)\ar[d]^{ch}\ar[r]& \projlim_rK_{n-1}(Y_r,Y)\ar[d]^\cong\\
\bigoplus_p\projlim_r\bb H^{2p-n}(Y_r,\Omega_{Y_r/K}^{\ge p})\ar[r]_-i& \bigoplus_pF^pH^{2p-n}_\sub{dR}(Y/K)\ar[r]& \bigoplus_p\projlim_r\bb H^{2p-(n-1)}(Y_r,\Omega_{(Y_r,Y)/K}^{\ge p})&
}\] Moreover, the image of the map $i$ is $\bigoplus_pF^p_\sub{flat}H_\sub{dR}^{2p-n}(Y/K)$.
\end{theorem}
\begin{proof}\rm
Goodwillie's Chern character from $K$-theory to negative cyclic homology gives a commutative diagram of pro abelian groups with exact rows:
\begin{equation}
\xymatrix{
\{K_n(Y_r)\}_r\ar[r]\ar[d]^{ch}& K_n(Y)\ar[d]^{ch}\ar[r]& \{K_{n-1}(Y_r,Y)\}_r\ar[d]^\cong\\
\{HN_n^\bb Q(Y_r)\}_r\ar[r]_i& HN_n^\bb Q(Y)\ar[r]& \{HN_{n-1}^\bb Q(Y_r,Y)\}_r
}\label{chern_character}
\end{equation}
where the right vertical arrow is an isomorphism by Goodwillie's theorem \cite{Goodwillie1986}.

As explained in section \ref{section_HKR}, with $k=\bb Q$, the classical from of the HKR theorem implies that $HN_n^\bb Q(Y)\cong \bigoplus_p\bb H^{2p-n}(Y,\Omega_{Y/\bb Q}^{\ge p})$, which equals $\bigoplus_pF^pH^{2p-n}_\sub{dR}(Y/K)$ by the fact that $K$ is algebraic over $\bb Q$ and by degeneration of the Hodge-to-de-Rham spectral sequence for $Y$. Secondly, corollary \ref{corollary_pro_HKR}, again with $k=\bb Q$, states that $\{HN_n^\bb Q(Y_r)\}_r\cong \bigoplus_p\{\bb H^{2p-n}(Y_r,\Omega_{Y_r/\bb Q}^{\ge p})\}_r$. It follows that \[\textstyle\{HN_{n-1}^\bb Q(Y_r,Y)\}_r\cong\bigoplus_p\{\bb H^{2p-(n-1)}(Y_r,\Omega_{(Y_r,Y)/\bb Q}^{\ge p})\}_r.\] Since $K$ is algebraic over $\bb Q$, it does not matter whether any of the aforementioned sheaves of K\"ahler differentials are taken over $K$ or over $\bb Q$.

To complete the proof it now remains only to replace the pro abelian groups occurring in (\ref{chern_character}) by their limits, which is done as follows. Since $K$ is algebraic over $\bb Q$ and $Y$ is proper over $K$, all the aforementioned hypercohomology groups are finite-dimensional $K$-spaces; thus the bottom row of (\ref{chern_character}) remains exact after taking $\projlim_r$. As for the top row of (\ref{chern_character}), the right is an inverse system of finite dimensional $K$-spaces (thanks to the isomorphism), whence both it and the inverse system $\{\op{Im}(K_{n+1}(Y_r,Y)\to K_n(Y_r))\}_r=\{\ker(K_n(Y_r)\to K_n(Y))\}_r$ are Mittag-Leffler. It easily follows that the sequence $\projlim_r K_n(Y_r)\to K_n(Y)\to\projlim_rK_{n-1}(Y_r,Y)$ is exact, as required.

Finally, the description of the image of $i$ is exactly the previous lemma.
\end{proof}

Setting $n=0$, this completes the proof of commutative diagram (\ref{main_commutative_diagram}), hence also of Theorem \ref{main_theorem}.

We finish the paper with two remarks. Firstly, Goodwillie's Chern character $\op{ch}:K_n(Y)\to HN^\bb Q_n(Y)=\bigoplus_pF^pH^{2p-n}_\sub{dR}(Y/K)$ appearing in Theorems \ref{main_theorem} and \ref{proof_of_main_diagram} coincides with the usual Chern character of de Rham cohomology by \cite[Thm.~1]{Weibel1993}. Secondly, suppose that $n>0$. Then Esnault and Kerz have pointed out to me that the Chern characters $K_n(Y)\to F^pH_\sub{dR}^{2p-n}(Y/K)$ are zero for all $p$, by a weight argument over $\bb C$. So, assuming $K$ is algebraic over $\bb Q$, it follows at once from Theorem \ref{proof_of_main_diagram}, or essentially just from diagram (\ref{chern_character}), that the map $\projlim_rK_n(Y_r)\to K_n(Y)$ is surjective. We presume this is a known result but cannot provide a reference.

\section*{Acknowledgements}
This note originated as a letter, entitled ``Deformational Hodge conjecture v.s.~Pro HKR'', to H.~Esnault after a hospitable visit to the Freie Universit\"at Berlin in April 2013. During this visit she explained to me the precise details of the main result of \cite{BlochEsnaultKerz2012}, after which the similarity with the pro HKR theorems on which I was working at the time became clear to me. I would like to thank her, S.~Bloch, and M.~Kerz for their interest in this short proof of mine;  their own work \cite{BlochEsnaultKerz2013} on such deformational problems treats more general base fields $K$ by combining Theorem \ref{Pro_HKR} with an assumed Chow--K\"unneth decomposition of $Y$.

\small
\bibliographystyle{acm}
\bibliography{../Bibliography}
\comment{

}
\normalsize
\noindent Matthew Morrow\hfill {\tt morrow@math.uni-bonn.de}\\
Mathematisches Institut\hfill \url{http://www.math.uni-bonn.de/people/morrow/}\\\
Universit\"at Bonn\\
Endenicher Allee 60\\
53115 Bonn, Germany

\end{document}